\documentclass[12pt]{amsart} 
\usepackage{amsfonts,graphics,amsmath,amsthm,amsfonts,amscd,amssymb,amsmath,latexsym,multicol,euscript}
\usepackage{epsfig,url}
\usepackage{flafter}
\usepackage{hyperref}
\usepackage[UglyObsolete]{diagrams}

\diagramstyle[noPostScript]

\makeatletter

\def\jobis#1{FF\fi
  \def\preedicate{#1}%
  \edef\preedicate{\expandafter\strip@prefix\meaning\preedicate}%
  \edef\job{\jobname}%
  \ifx\job\preedicate
}

\makeatother

\if\jobis{proposal}%
 \def\try{subsection}%
\else
  \def\try{section}%
\fi

\theoremstyle{plain}
\newtheorem{theorem}{Theorem}[\try]

\newtheorem{lemma}[theorem]{Lemma}
\newtheorem{claim}[theorem]{Claim}

\newtheorem{definition-lemma}[theorem]{Definition-Lemma}

\newtheorem{definition}[theorem]{Definition}

\newtheorem{theorema}{Theorem}

\def\scr#1{\mathbf{\EuScript{#1}}}

\def\ideal#1.{I_{#1}}
\def\ring#1.{\mathcal {O}_{#1}}
\def\fring#1.{\hat{\mathcal {O}}_{#1}}
\def\proj#1.{\mathbb {P}(#1)}
\def\pr #1.{\mathbb {P}^{#1}}
\def\dpr #1.{\hat{\mathbb {P}}^{#1}}
\def\af #1.{\mathbb{A}^{#1}}
\def\Hz #1.{\mathbb{F}_{#1}}
\def\Hbz #1.{\overline{\mathbb {F}}_{#1}}
\def\fb#1.{\underset {#1} {\times}}
\def\rest#1.{\underset {\ \ring #1.} \to \otimes}
\def\au#1.{\operatorname {Aut}\,(#1)}
\def\deg#1.{\operatorname {deg } (#1)}
\def\pic#1.{\operatorname {Pic}\,(#1)}
\def\pico#1.{\operatorname{Pic}^0(#1)}
\def\picg#1.{\operatorname {Pic}^G(#1)}
\def\ner#1.{NS (#1)}
\def\rdown#1.{\llcorner#1\lrcorner}
\def\rfdown#1.{\lfloor{#1}\rfloor}
\def\rup#1.{\ulcorner{#1}\urcorner}
\def\rcup#1.{\lceil{#1}\rceil}
\def\cone#1.{\operatorname {NE}(#1)}
\def\mone#1.{\operatorname {NM}(#1)}
\def\ccone#1.{\overline{\operatorname {NE}}(#1)}
\def\cmone#1.{\overline{\operatorname {NM}}(#1)}
\def\none#1.{\operatorname {NF}(#1)}
\def\cnone#1.{\overline{\operatorname {NF}}(#1)}
\def\coef#1.{\frac{(#1-1)}{#1}}
\def\vit#1.{D_{\langle #1 \rangle}}
\def\mm#1.{\overline {M}_{0,#1}}
\def\H1#1.{H^1(#1,{\ring #1.})}
\def\ac#1.{\overline {\mathbb F}_{#1}}

\def\adj#1.{\frac {#1-1}{#1}}
\def\spn#1.{\overline{#1}}
\def\pek#1.#2.{\Cal P^{#1}(#2)}
\def\plk#1.#2.{\Cal P^{\leq #1}(#2)}
\def\ev#1.{\operatorname{ev_{#1}}}
\def\ilist#1.{{#1}_1,{#1}_2,\ldots}
\def\bminv#1.{(\nu_1,s_1;\nu_2,s_2;\dots ;\nu_{#1},s_{#1};\nu_{r+1})}
\def\zinv#1.{(\nu_1,s_1;\nu_2,s_2;\dots ;\nu_{#1},s_{#1};0)}
\def\iinv#1.{(\nu_1,s_1;\nu_2,s_2;\dots ;\nu_{#1},s_{#1};\infty)}
\def\scr#1.{\mathbf{\EuScript{#1}}}
\def\mg#1.{\overline {M}_{#1}}
\def\inter#1.{\underset #1{\cdot}}
\def\WDiv{\operatorname{WDiv}}

\def\llist#1.#2.{{#1}_1,{#1}_2,\dots,{#1}_{#2}}
\def\ulist#1.#2.{{#1}^1,{#1}^2,\dots,{#1}^{#2}}
\def\lomitlist#1.#2.{{#1}_1,{#1}_2,\dots,\hat {{#1}_i}, \dots, {#1}_{#2}}
\def\lomitlistz#1.#2.{{#1}_0,{#1}_1,\dots,\hat {{#1}_i}, \dots, {#1}_{#2}}
\def\loc#1.#2.{\Cal O_{#1,#2}}
\def\fderiv#1.#2.{\frac {\partial #1}{\partial #2}}
\def\deriv#1.#2.{\frac {d #1}{d #2}}
\def\map#1.#2.{#1 \longrightarrow #2}
\def\rmap#1.#2.{#1 \dasharrow #2}
\def\emb#1.#2.{#1 \hookrightarrow #2}
\def\non#1.#2.{\text {Spec }#1[\epsilon]/(\epsilon)^{#2}}
\def\Hi#1.#2.{\text {Hilb}^{#1}(#2)}
\def\sym#1.#2.{\operatorname {Sym}^{#1}(#2)}
\def\Hb#1.#2.{\text {Hilb}_{#1}(#2)}
\def\Hm#1.#2.{\Hom_{#1}(#2)}
\def\prd#1.#2.{{#1}_1\cdot {#1}_2\cdots {#1}_{#2}}
\def\Bl #1.#2.{\operatorname {Bl}_{#1}#2}
\def\pl #1.#2.{#1^{\otimes #2}}
\def\mgn#1.#2.{\overline {M}_{#1,#2}}
\def\ialist#1.#2.{{#1}_1 #2 {#1}_2, #2\dots}
\def\pair#1.#2.{\langle #1, #2\rangle}
\def\gproj#1.#2.{\mathbb{P}_{#1}(#2)}
\def\gpr #1.#2.{\mathbb{P}^{#1}_{#2}}
\def\gaf #1.#2.{\mathbb{A}^{#1}_{#2}}
\def\vandermonde#1.#2.{\left|
\begin{matrix}
1 & 1 & 1 & \dots & 1\\
{#1}_1 & {#1}_2 & {#1}_3 & \dots & {#1}_{#2}\\
{#1}_1^2 & {#1}_2^2 & {#1}_3^2 & \dots & {#1}_{#2}^2\\
\vdots & \vdots & \vdots & \ddots & \vdots\\
{#1}_1^{#2-1} & {#1}_2^{#2-1} & {#1}_2^{#2-1} & \dots & {#1}_{#2}^{#2-1}\\
\end{matrix}
\right|
}
\def\vandermondet#1.#2.{\left|
\begin{matrix}
1 & {#1}_1   & {#1}_1^2 & \dots & {#1}_1^{#2-1}\\
1 & {#1}_2   & {#1}_2^2 & \dots & {#1}_2^{#2-1}\\
1 & {#1}_3   & {#1}_3^2 & \dots & {#1}_3^{#2-1}\\
\vdots & \vdots & \vdots & \ddots & \vdots\\
1 & {#1}_{#2}& {#1}_{#2}^2 & \dots & {#1}_{#2}^{#2-1}\\
\end{matrix}
\right|
}
\def\gr#1.#2.{\mathbb{G}(#1,#2)}

\def\alist#1.#2.#3.{{#1}_1 #2 {#1}_2 #2\dots #2 {#1}_{#3}}
\def\zlist#1.#2.#3.{#1_0 #2 #1_1 #2\dots #2 #1_{#3}}
\def\lomitlist30#1.#2.#3.{{#1}_0,{#1}_1 #2 \dots #2\hat {{#1}_i} #2\dots #2 {#1}_{#3}}
\def\lmap#1.#2.#3.{#1 \overset{#2}{\longrightarrow} #3}
\def\mes#1.#2.#3.{#1 \longrightarrow #2 \longrightarrow #3}
\def\ses#1.#2.#3.{0\longrightarrow #1 \longrightarrow #2 \longrightarrow #3 \longrightarrow 0}
\def\les#1.#2.#3.{0\longrightarrow #1 \longrightarrow #2 \longrightarrow #3}
\def\res#1.#2.#3.{#1 \longrightarrow #2 \longrightarrow #3\longrightarrow 0}
\def\Hi#1.#2.#3.{\text {Hilb}^{#1}_{#2}(#3)}
\def\ten#1.#2.#3.{#1\underset {#2}{\otimes} #3}
\def\lomitlist30#1.#2.#3.{{#1}_0 #2 {#1}_1 #2 \dots #2 \hat {{#1}_i} #2 \dots #2 {#1}_{#3}}
\def\mderiv#1.#2.#3.{\frac {d^{#3} #1}{d #2^{#3}}}
\def\ggr#1.#2.#3.{\mathbb{G}_{#1}(#2,#3)}

\def\Hom{\operatorname{Hom}}

\def\sp{\operatorname{Spec}}
\def\Proj{\operatorname{Proj}}

\def\dim{\operatorname{dim}}

\def\deg{\operatorname{deg}}

\def\im{\operatorname{Im}}

\def\mult{\operatorname{mult}}

\def\fix{\operatorname{Fix}}

\def\rest{\operatorname{res}}

\def\adj{\operatorname{adj}}

\def\e{\Cal E}

\def\e1{E_1}
\def\e2{E_2}

\def\ds{\displaystyle}

\def\qle{\sim_{\mathbb Q}}
\def\bd{_{\bullet}}

\def\mapdown#1{\big\downarrow\rlap{$\vcenter
{\hbox{$\scriptstyle#1$}}$}}

\def\mapse#1{
{\vcenter{\hbox{$\mathop{\smash{\raise1pt\hbox{$\diagdown$}\!\lower7pt
\hbox{$\searrow$}}\vphantom{p}}\limits_{#1}\vphantom{\mapdown{}}$}}}}

\def\VR#1.{height#1pt&\omit&&\omit&&\omit&&\omit&&\omit&\cr}

\def\VRT#1.{height#1pt&\omit&&\omit&\cr}

\begin{document}
\title{Existence of minimal models for varieties of log general type II}  
\date{\today}
\author{Christopher D. Hacon} 
\address{Department of Mathematics \\  
University of Utah\\  
155 South 1400 East\\
JWB 233\\
Salt Lake City, UT 84112, USA}
\email{hacon@math.utah.edu}
\author{James M\textsuperscript{c}Kernan} 
\address{Department of Mathematics\\ 
University of California at Santa Barbara\\ 
Santa Barbara, CA 93106, USA} 
\email{mckernan@math.ucsb.edu}
\address{Department of Mathematics\\ 
MIT\\ 
77 Massachusetts Avenue\\
Cambridge, MA 02139, USA} 
\email{mckernan@math.mit.edu}

\thanks{The first author was partially supported by NSF research grant no: 0456363 and an
  AMS Centennial fellowship and the second author was partially supported by NSA grant no:
  H98230-06-1-0059 and NSF grant no: 0701101.  We would like to thank F. Ambro, C. Birkar,
  P. Cascini, J. A. Chen, A. Corti, O.  Fujino, S. Keel, \and J. Koll\'ar for valuable
  suggestions.}

\begin{abstract} Assuming finite generation in dimension $n-1$, we prove that pl-flips 
exist in dimension $n$.
\end{abstract}

\maketitle

\tableofcontents

\section{Introduction}
\label{s_int}

This is the second of two papers whose purpose is to establish:
\begin{theorem}\label{t_finite} The canonical ring
$$
R(X,K_X)=\bigoplus_{m\in\mathbb{N}}H^0(X,\ring X.(mK_X)),
$$
is finitely generated for every smooth projective variety $X$.  
\end{theorem}

Note that Siu has announced a proof of finite generation for varieties of general type,
using analytic methods, see \cite{Siu06}.  

Our proof relies on the ideas and techniques of the minimal model program and roughly
speaking in this paper we will show that finite generation in dimension $n-1$ implies the
existence of flips in dimension $n$.  More precisely, assuming the following:
\setcounter{theorema}{5}
\begin{theorema}\label{t_ezd} Let $\pi\colon\map X.Z.$ be a projective morphism to a 
normal affine variety.  Let $(X,\Delta=A+B)$ be a $\mathbb{Q}$-factorial kawamata log
terminal pair of dimension $n$, where $A\geq 0$ is an ample $\mathbb{Q}$-divisor and
$B\geq 0$.  If $K_X+\Delta$ is pseudo-effective, then
\begin{enumerate} 
\item The pair $(X,\Delta)$ has a log terminal model $\mu\colon\rmap X.Y.$.  
In particular if $K_X+\Delta$ is $\mathbb{Q}$-Cartier then the log canonical ring 
$$
R(X,K_X+\Delta)=\bigoplus_{m\in\mathbb{N}}H^0(X,\ring X.(\rdown m(K_X+\Delta).)),
$$
is finitely generated. 
\item Let $V\subset \WDiv _{\mathbb{R}}(X)$ be the vector space spanned by the components
of $\Delta$.  Then there is a constant $\delta>0$ such that if $G$ is a prime divisor
contained in the stable base locus of $K_X+\Delta$ and $\Xi\in \mathcal L _{A}(V)$ such
that $\|\Xi-\Delta\|<\delta$, then $G$ is contained in the stable base locus of $K_X+\Xi$.
\item Let $W\subset V$ be the smallest affine subspace of $\WDiv _{\mathbb{R}}(X)$
containing $\Delta$, which is defined over the rationals.  Then there is a constant
$\eta>0$ and a positive integer $r>0$ such that if $\Xi\in W$ is any divisor and $k$ is
any positive integer such that $\|\Xi-\Delta\|<\eta$ and $k(K_X+\Xi)/r$ is Cartier, then
every component of $\fix (k(K_X+\Xi))$ is a component of the stable base locus of
$K_X+\Delta$.
\end{enumerate} 
\end{theorema}
we prove the existence of pl-flips:
\setcounter{theorema}{0}
\begin{theorema}\label{t_existence} Pl-flips exist in dimension $n$.  
\end{theorema} 
that is, we prove:
\begin{theorem}\label{t_m}  Theorem~\ref{t_ezd}$_{n-1}$ implies Theorem~\ref{t_existence}$_n$.
\end{theorem} 

With the results of \cite{BCHM06}, \eqref{t_m} completes the proof of \eqref{t_finite}.

The main ideas used in this paper have their origins in the work of Shokurov on the
existence of flips \cite{Shokurov03} together with the use of the extension theorem of
\cite{HM05b} which in turn was inspired by the work of Kawamata, Siu and Tsuji (cf.
\cite{Kawamata99}, \cite{Siu98} and \cite{Tsuji99}).  For further history about the
details of this problem see \cite[\S 2.1]{Corti05}.

In this paper, however we do not make use of the concept of \lq\lq asymptotic
saturation\rq\rq\ introduced by Shokurov, and in fact we prove a more general result which
does not require the relative weak log Fano condition (see also \cite{Ambro06}).

Further treatments of the results of this paper may be found in \cite{Ambro06} and
\cite{HM07} (which follows Shokurov's approach more explicitly).

We now turn to a more detailed description of the results and techniques used in this
paper.  Recall the following:
\begin{definition}\label{d_pl} Let $(X,\Delta )$ be a purely log terminal pair and 
$f\colon\map X.Z.$ be a projective morphism of normal varieties.  Then $f$ is a
\textbf{pl-flipping contraction} if $\Delta $ is a $\mathbb{Q}$-divisor and
\begin{enumerate}
\item $f$ is small, of relative Picard number one, 
\item $-(K_X+\Delta)$ is $f$-ample,
\item $X$ is $\mathbb{Q}$-factorial, 
\item $S=\rdown\Delta.$ is irreducible and $-S$ is $f$-ample. 
\end{enumerate}

The \textbf{flip} of a pl-flipping contraction $f\colon\map X.Z.$ is a small projective
morphism $g\colon\map Y.Z.$ of relative Picard number one, such that $K_Y+\Gamma$ is
$g$-ample, where $\Gamma$ is the strict transform of $\Delta$.
\end{definition}

The flip $g$ is unique, if it exists at all, and it is given by
$$
Y=\Proj_Z \mathfrak{R} \qquad \text{where} \qquad \mathfrak{R}=\bigoplus _{m\in \mathbb{N}\,:\,k|m}f_*\ring X.(m(K_X+\Delta)),
$$
and $k$ is any positive integer such that $k(K_X+\Delta)$ is integral.  Therefore, in
order to prove the existence of pl-flips, it suffices to show that $\mathfrak{R}$ is a
finitely generated $\ring Z.$-algebra.  Since this problem is local over $Z$, we may
assume that $Z=\sp A$ is affine and it suffices to prove that
$$
R(X,k(K_X+\Delta))=\bigoplus _{m\in \mathbb{N}\,:\,k|m}H^0(X,\ring X.(m(K_X+\Delta))),
$$
is a finitely generated $A$-algebra.  It is then natural to consider the restricted
algebra
$$
R_S(X,k(K_X+\Delta))=\im\left(\map R(X,k(K_X+\Delta)).R(S,k(K_S+\Omega)).\right),
$$
whose graded pieces correspond to the images of the restriction homomorphisms
$$
\map {H^0(X,\ring X.(m(K_X+\Delta)))}.{H^0(S,\ring S.(m(K_S+\Omega)))}.,
$$
where $m=kl$ is divisible by $k$ and $\Omega$ is defined by the adjunction formula 
$$
(K_X+\Delta)|_S=K_S+\Omega,
$$
and $k(K_X+\Delta)$ is Cartier.  Shokurov has shown, cf. \eqref{t_restricted}, that the
algebra $R(X,k(K_X+\Delta))$ is finitely generated if and only if the restricted algebra
is finitely generated.

Now, if the natural inclusion 
$$
R_S(X,k(K_X+\Delta))\subset R(S,k(K_S+\Omega)),
$$ 
were an isomorphism, then \eqref{t_m} would follow from (1) of
Theorem~\ref{t_ezd}$_{n-1}$.  In fact the pair $(S,\Omega)$ is kawamata log terminal,
$\dim S=\dim X-1=n-1$ and since $f|_S$ is birational, $\Omega$ is automatically big so
that, by a standard argument, (1) of Theorem~\ref{t_ezd}$_{n-1}$ applies and
$R(S,k(K_S+\Omega))$ is finitely generated.  \eqref{t_restricted} also implies
$\mathfrak{R}$ is finitely generated.

Unluckily this is too much to hope for.  However it does suggest that one should
concentrate on the problem of lifting sections and the main focus of this paper is to
prove the extension result \eqref{t_lift}.  In fact \eqref{t_m} is a straightforward
consequence of \eqref{t_lift}.

To fix ideas, let us start with an example where we cannot lift sections.  Let $X$ be the
blow up of $\pr 2.$ at a point $o$, with exceptional divisor $E$.  Let $S$ be the strict
transform of a line through $o$, let $L_1$, $L_2$ and $L_3$ be the strict transforms of
general lines in $\pr 2.$, let $p=E\cap S$ and let $p_i=L_i\cap S$.  Then the pair
$$
(X,\Delta =S+(2/3)(E+L_1+L_2+L_3)),
$$ 
is purely log terminal but the homomorphism
\[
\map {H^0(\ring X.(3l(K_X+\Delta )))}.{H^0(\ring S.(3l(K_S+\Omega )))}.\simeq H^0 (\ring {\pr 1.}.(2l)),
\]
is never surjective, where $\Omega=(\Delta-S)|_S=2/3(p+p_1+p_2+p_3)$ and $l$ is a positive
integer.  The problem is that the stable base locus of $K_X+\Delta$ contains $E$ and yet
$|3(K_S+\Omega)|$ is base point free.  Notice, however, that
$$
|3l(K_X+\Delta )|_S=|3l(K_S+\Theta )|+3l(\Omega-\Theta),
$$
where $\Theta=(2/3)(p_1+p_2+p_3)$ is obtained from $\Omega$ by throwing away $p$.  In
other words, $\Theta$ is obtained from $\Omega$ by removing some part of each components
contained in the stable base locus of $K_X+\Delta$.

Returning to the general setting, one may then hope that the restricted algebra
$R_S(S,l(K_X+\Delta))$ is given by an algebra of the form $R(S,l(K_S+\Theta))$ for some
kawamata log terminal pair $(S,\Theta)$ where $0\leq\Theta\leq\Omega$ is a
$\mathbb{Q}$-divisor obtained from $\Omega$ by subtracting components of $\Omega$
contained in the stable base locus of $K_X+\Delta$.  We will now explain how this may be
achieved.  The tricky thing is to determine exactly how much of the stable base locus to
throw away.

It is not hard to reduce to the following situation: $\pi\colon\map X.Z.$ is a projective
morphism to a normal affine variety $Z$, where $(X,\Delta=S+A+B)$ is a purely log terminal
pair of dimension $n$, $S=\rdown\Delta.$ is irreducible, $X$ and $S$ are smooth, $A\geq 0$
is an ample $\mathbb{Q}$-divisor, $B\geq 0$, $(S,\Omega=(\Delta-S)|_S)$ is canonical and
the stable base locus of $K_X+\Delta$ does not contain $S$.  

Let 
$$
\Theta _m=\Omega -\Omega \wedge F_m\qquad \text{where}\qquad F_m=\fix (|m(K_X+\Delta)|_S)/m,
$$
and $m(K_X+\Delta)$ is Cartier.  Then $m(\Omega-\Theta_m)$ is the biggest divisor
contained in $\fix(|m(K_X+\Delta)|_S)$ such that $0\leq\Theta_m\leq\Omega$.
It follows that
\[
\label{e_sup} |m(K_S+\Theta _m)|+m(\Omega-\Theta_m)\supset |m(K_X+\Delta)|_S. \tag{$\supset$}
\] 
A simple consequence of the main lifting result \eqref{t_lift} of this paper implies that
this tautological inclusion \eqref{e_sup} is actually an equality,
\[
\label{e_eq} |m(K_S+\Theta_m)|+m(\Omega-\Theta_m)=|m(K_X+\Delta)|_S. \tag{$=$}
\] 
A technical, but significant, improvement on the proof of the existence of flips which
appears in \cite{HM07} is that the statement of \eqref{e_eq} and of \eqref{t_lift}
involves only linear systems and divisors on $X$, even though the proof of \eqref{t_lift}
involves passing to a higher model.  The key point is that since $(S,\Omega)$ is
canonical, it suffices to keep track only of the fixed divisor on $S$ and not of the whole
base locus.

To prove \eqref{e_eq} we use the method of multiplier ideal sheaves.  In fact the main
point is to establish an inclusion of multiplier ideal sheaves, \eqref{t_multiplier}.  A
proof of \eqref{t_multiplier} appeared originally in \cite{HM05b}.  We chose to include a
proof of this result for the convenience of the reader and we decided to use notation
closer to the well established notation used in \cite{Lazarsfeld04b}.  Note however that
the multiplier ideal sheaves we use, see \eqref{d_variant}, must take into account the
divisor $\Delta$ (for example consider the case worked out above) and the fact that
$(S,\Omega)$ is canonical.

In fact \eqref{e_eq} follows from the MMP.  Indeed, if one runs $f\colon\rmap X.Y.$ the
$(K_X+\Delta)$-MMP, almost by definition this will not change the linear systems
$|m(K_X+\Delta)|$.  Since $K_Y+\Gamma=K_X+f_*\Delta$ is nef, one can lift sections on $Y$
from the strict transform $T$ of $S$, by an easy application of Kawamata-Viehweg
vanishing.  In general, however, the linear systems $|m(K_T+g_*\Theta)|$ are bigger than
the linear systems $|m(K_S+\Theta)|$, since the induced birational map $g\colon\rmap S.T.$
might extract some divisors.  However any such divisor must have log discrepancy at most
one, so this cannot happen, almost by definition, if $K_S+\Theta$ is canonical.

In order to establish that $R_S(X,k(K_X+\Delta))$ is finitely generated, cf.
\eqref{t_rational}, and thereby to finish the proof of \eqref{t_m}, it is necessary and
sufficient to show that $\Theta=\lim (\Theta _{m!}/m!)$ is rational (the seemingly strange
use of factorials is so that we can use limits rather than limsups).  At this point we
play off two facts.  The first is that since we are assuming that Theorem~\ref{t_ezd} holds on
$S$, if $m>0$ is sufficiently divisible and $\Phi$ is an appropriately chosen $\mathbb
Q$-divisor sufficiently close to $\Theta$, then the base locus of $|m(K_S+\Phi)|$ and the
stable base locus of $K_S+\Theta$ are essentially the same (basically because $K_S+\Theta$
and $K_S+\Phi$ share a log terminal model $\mu\colon\rmap S.S'.$ and these two sets of
divisors are precisely the divisors contracted by $\mu$).  The second is that using
\eqref{t_squeeze}, \eqref{t_lift} is slightly stronger than \eqref{e_eq}; one is allowed
to overshoot $\Theta_m$ by an amount $\epsilon/m$, where $\epsilon>0$ is fixed.  (It seems
worth pointing out that \eqref{t_squeeze} seems to us a little mysterious.  In particular,
unlike \eqref{t_lift}, we were unable to show that this result follows from the MMP.)

More precisely, since the base locus of $|m(K_S+\Theta_m)|$ contains no components of
$\Theta_m$, by (2) of Theorem~\ref{t_ezd} it follows that the stable base locus of $K_S+\Theta$
contains no components of $\Theta$.  If $\Theta$ is not rational, then by Diophantine
approximation there is a $\mathbb{Q}$-divisor $0\leq \Phi\leq \Omega$ very close to
$\Theta$ and an integer $k>0$ such that $k\Phi$ is integral and
$\mult_G\Phi>\mult_G\Theta$, for some prime divisor $G$.  By \eqref{t_lift}, it actually
follows that
\[
|k(K_S+\Phi)|+k(\Omega-\Phi)=|k(K_X+\Delta )|_S.
\]
The condition $\mult_G\Phi>\mult_G\Theta$ ensures that $G$ is a component of $\fix
(k(K_S+\Phi))$, and hence of the stable base locus of $K_S+\Phi$.  But then $G$ is a
component of $\Theta$ and of the stable base locus of $\Theta$.  This is the required
contradiction.

\section{Notation and conventions}
\label{s_notation}

We work over the field of complex numbers $\mathbb{C}$.  Let $X$ be a normal variety.  A
$$
\left.
\begin{array}{r}
\text{\textit{(integral) divisor}} \\
\text{\textit{$\mathbb{Q}$-divisor}} \\
\text{\textit{$\mathbb{R}$-divisor}} \\
\end{array}
\right\}
\quad \text{is a}
\quad
\left \{
\begin{array}{l}
\text{$\mathbb{Z}$-linear}\\
\text{$\mathbb{Q}$-linear}\\
\text{$\mathbb{R}$-linear,}\\
\end{array}
\right.
$$
combination of prime divisors.  Given an integral Weil divisor $D$, we let
$$
R(X,D)=\bigoplus _{m\in\mathbb{N}} H^0(X,\ring X.(mD)).  
$$ 
Set
\begin{align*} 
\WDiv_{\mathbb{Q}}(X)&=\ten \WDiv(X).\mathbb{Z}.\mathbb{Q}. \\ 
\WDiv_{\mathbb{R}}(X)&=\ten \WDiv (X).\mathbb{Z}.\mathbb{R}.,
\end{align*} 
where $\WDiv (X)$ is the group of Weil divisors on $X$.  The definitions below for
$\mathbb{R}$-divisors reduce to the usual definitions for $\mathbb{Q}$-divisors and
integral divisors, see \cite{BCHM06}.  Note that the group of $\mathbb{R}$-divisors forms
a vector space, with a canonical basis given by the prime divisors.  If $C=\sum c_iB_i$
and $D=\sum d_iB_i$, where $B_i$ are distinct prime divisors, then we write $D\geq 0$ if $d_i\geq
0$ and we will denote by
\begin{align*} 
\|C\|   &= \max_i c_i             & C\wedge D  &= \sum_i\min\{c_i,d_i\}B_i \\
\rdown C. &= \sum _i \rdown c_i. B_i  & \{C\}      &= C-\rdown C. . 
\end{align*}

Two $\mathbb{R}$-divisors $C$ and $D$ are 
$$
\left.
\begin{array}{rl}
\text{linearly equivalent,} & C\sim D\\
\text{$\mathbb{Q}$-linearly equivalent,} & C\sim_{\mathbb{Q}}D\\
\text{$\mathbb{R}$-linearly equivalent,} & C\sim_{\mathbb{R}}D\\
\end{array}
\right\}
\quad \text{if $C-D$ is a}
\quad
\left \{
\begin{array}{l}
\text{$\mathbb{Z}$-linear}\\
\text{$\mathbb{Q}$-linear}\\
\text{$\mathbb{R}$-linear,}\\
\end{array}
\right.
$$
combination of principal divisors.  Note that if $C\sim_{\mathbb{Q}}D$ then $mC\sim mD$
for some positive integer $m$, but this fails in general for $\mathbb{R}$-linear
equivalence.  Note also that if two $\mathbb{Q}$-divisors are $\mathbb{R}$-linearly
equivalent then they are in fact $\mathbb{Q}$-linearly equivalent, but that two integral
divisors might be $\mathbb{Q}$-linearly equivalent without being linearly equivalent.  Let
\begin{align*} 
|D|&=\{\,C\in \WDiv (X)\,|\, C\geq 0\, ,\, C \sim D \,\} \\ 
|D|_{\mathbb{Q}}&=\{\,C\in \WDiv _{\mathbb{Q}}(X)\,|\, C\geq 0\, ,\, C \sim_{\mathbb{Q}} D \,\} \\ 
|D|_{\mathbb{R}}&=\{\,C\in \WDiv_{\mathbb{R}} (X)\,|\, C\geq 0\, ,\, C \sim_{\mathbb{R}} D \,\}.
\end{align*} 
If $T$ is a subvariety of $X$, not contained in the base locus of $|D|$, then $|D|_T$
denotes the image of the linear system $|D|$ under restriction to $T$.  If $D$ is an
integral divisor, $\fix(D)$ denotes the fixed divisor of $D$ so that
$|D|=|D-\fix(D)|+\fix(D)$ where the base locus of $|D-\fix (D)|$ contains no divisors.
More generally $\fix(V)$ denotes the fixed divisor of the linear system $V$.

The \textit{stable base locus} of $D$, denoted by $\mathbf{B}(D)$, is the intersection of
the support of the elements of $|D|_{\mathbb{R}}$ (if $|D|_{\mathbb{R}}$ is empty then by convention the
stable base locus is the whole of $X$).  The \textit{stable fixed divisor} is the
divisorial support of the stable base locus.  The \textit{augmented stable base locus} of
$D$, denoted by $\mathbf{B}_+(D)$, is given by the stable base locus of $D-\epsilon A$ for
some ample divisor $A$ and any rational number $0<\epsilon \ll 1$.  The \textit{diminished
  stable base locus} is defined by
$$
\mathbf{B}_-(D)=\bigcup_{\epsilon>0} \mathbf{B}(D+\epsilon A).
$$
In particular we have
$$
\mathbf{B}_-(D)\subset \mathbf{B}(D)\subset \mathbf{B}_+(D).
$$

An \textit{$\mathbb{R}$-Cartier divisor} $D$ is an $\mathbb{R}$-linear combination of
Cartier divisors.  An $\mathbb{R}$-Cartier divisor $D$ is \textit{nef} if $D\cdot
\Sigma\geq 0$ for any curve $\Sigma\subset X$.  An $\mathbb{R}$-Cartier divisor $D$ is
\textit{ample} if it is $\mathbb{R}$-linearly equivalent to a positive linear combination
of ample divisors (in the usual sense).  An $\mathbb{R}$-Cartier divisor $D$ is
\textit{big} if $D \sim_{\mathbb{R}} A+B$, where $A$ is ample and $B\geq 0$.  A
$\mathbb{Q}$-Cartier divisor $D$ is a \textit{general ample $\mathbb{Q}$-divisor} if there
is an integer $m>0$ such that $mD$ is very ample and $mD\in |mD|$ is very general.

A \textit{log pair} $(X,\Delta)$ is a normal variety $X$ and $\mathbb{R}$-Weil divisor
$\Delta\geq 0$ such that $K_X+\Delta$ is $\mathbb{R}$-Cartier.  We say that a log pair
$(X,\Delta)$ is \textit{log smooth}, if $X$ is smooth and the support of $\Delta$ is a
divisor with global normal crossings.  A projective birational morphism $g\colon\map Y.X.$
is a \textit{log resolution} of the pair $(X,\Delta )$ if $X$ is smooth and the inverse
image of $\Delta$ union the exceptional locus is a divisor with global normal crossings.
Note that in the definition of log resolution we place no requirement that the
indeterminacy locus of $g$ is contained in the locus where the pair $(X,\Delta)$ is not
log smooth.  If $V$ is a linear system on $X$, a \textit{log resolution of $V$ and
  $(X,\Delta)$} is a log resolution of the pair $(X,\Delta)$ such that if $|M|+F$ is the
decomposition of $g^*V$ into its mobile and fixed parts, then $|M|$ is base point free and
$F$ union the exceptional locus union the strict transform of $\Delta$ is a divisor with
simple normal crossings support.  If $g$ is a log resolution, then we may write
$$
K_Y+\Gamma=g^*(K_X+\Delta)+E,
$$
where $\Gamma\geq 0$ and $E\geq 0$ have no common components, $g_*\Gamma =\Delta$ and $E$
is $g$-exceptional.  Note that this decomposition is unique.  The \textit{log discrepancy}
of a divisor $F$ over $X$ 
$$
a(X,\Delta,F)=1+\mult_F(E-\Gamma).
$$
Note that with this definition, a component $F$ of $\Delta$ with coefficient $b$ has log
discrepancy $1-b$.  The log discrepancy does not depend on the choice of model $Y$, so
that the log discrepancy is also a function defined on valuations.  A \textit{log
  canonical place} is any valuation of log discrepancy at most zero and the centre of a
log canonical place is called a \textit{log canonical centre}.  Note that every divisor on
$X$ is by definition a canonical centre, so the only interesting canonical centres are of
codimension at least two.

The pair $(X,\Delta)$ is \textit{kawamata log terminal} if there are no log canonical
centres.  We say that the pair $(X,\Delta)$ is \textit{purely log terminal} (respectively
\textit{canonical} or \textit{terminal}) if the log discrepancy of any exceptional divisor
is greater than zero (respectively at least one or greater than one).  We say that the
pair is \textit{divisorially log terminal} if there is a log resolution $g\colon\map Y.X.$
such that all exceptional divisors $E\subset Y$ have log discrepancy greater than zero.

\section{Preliminary results}
\label{s_preliminary}
 
In this section we recall several results about finitely generated algebras and in
particular we will give a proof of Shokurov's result that the pl-flip exists if and only
if the restricted algebra is finitely generated.

\begin{definition} Let $X$ be a normal variety, $S$ be a prime divisor and $B$ an integral
Weil divisor which is $\mathbb Q$-Cartier and whose support does not contain $S$.  The
\textbf{restricted algebra} $R_S(X,B)$ is the image of the homomorphism $\map
R(X,B).R(S,B|_S).$.
\end{definition}

We remark that as $B$ is $\mathbb Q$-Cartier then $B|_{S}$ is a well defined
$\mathbb{Q}$-Cartier divisor on $S$.

\begin{theorem}\label{t_restricted} Let $f\colon\map X.Z.$ be a pl-flipping contraction 
with respect to $(X,\Delta)$.  Pick an integer $k$ such that $k(K_X+\Delta)$ is Cartier.

 Then 
\begin{enumerate} 
\item The flip of $f$ exists if and only if the flip of $f$ exists locally over $Z$.  
\item If $Z=\sp A$ is affine then the flip $f^{+}\colon\map X^+.Z.$ exists if and only if
the restricted algebra $R_S(X,k(K_X+\Delta))$ is a finitely generated $A$-algebra.
\end{enumerate} 
\end{theorem}

We start with the following well known result:
\begin{lemma}\label{l_truncation} Let $R$ be a graded algebra which is an integral 
domain and let $d$ be a positive integer.

Then $R$ is finitely generated if and only if $R_{(d)}$ is finitely generated.
\end{lemma}
\begin{proof} Suppose that $R$ is finitely generated.  It is easy to write down an action
of the cyclic group $\mathbb{Z}_d$ on $R$ so that the invariant ring is $R_{(d)}$.  Thus
$R_{(d)}$ is finitely generated by the Theorem of E. Noether which states that the ring of
invariants of a finitely generated ring under the action of a finite group is finitely
generated.

Suppose now that $R_{(d)}$ is finitely generated. Let $f\in R_i$.  Then $f$ is a root of
the monic polynomial $x^d-f^d\in R_{(d)}[x]$.  It follows that $R$ is integral over
$R_{(d)}$ and the result follows by another Theorem of E. Noether on finiteness of
integral closures.
\end{proof}

\begin{lemma}\label{l_restricted} Let $S$ be a normal prime divisor on $X$ and let 
$B$ an integral Weil divisor which is $\mathbb Q$-Cartier and whose support does not contain $S$.
\begin{itemize} 
\item If $R(X,B)$ is finitely generated then $R_S(X,B)$ is finitely generated. 
\item If $S\sim B$ and $R_S(X,B)$ is finitely generated then $R(X,B)$ is finitely
generated.
\end{itemize} 
\end{lemma}
\begin{proof} Since there is a surjective homomorphism $\phi\colon\map R(X,B).R_S(X,B).$, it is
clear that if $R(X,B)$ is finitely generated then $R_S(X,B)$ is finitely generated.

Suppose now that $R_S(X,B)$ is finitely generated and $S\sim B$.  Then there is a rational
function $g_1$ such that $(g_1)=S-B$.  If we consider the elements of $R(X,B)_m$ as
rational functions, then a rational function $g$ belongs to $R(X,B)_m$ if and only if
$(g)+mB\geq 0$.  But if $g$ is in the kernel of $\phi$, then there is a divisor $S'\geq 0$
such that $(g)+mB=S+S'$.  It follows that $(g/g_1)+(m-1)B=S'$ so that $g/g_1=h\in
R(X,B)_{m-1}$.  But then the kernel of $\phi$ is the principal ideal generated by
$g_1$.  \end{proof}

\begin{proof}[Proof of \eqref{t_restricted}] It is well known that the flip
$f^{+}\colon\map X^+.Z.$ exists if and only if the sheaf of graded $\ring Z.$-algebras
$$
\bigoplus_{m\in\mathbb{N}\,:\,k|m}f_*\ring X.(m(K_X+\Delta)),
$$
is finitely generated, cf.~\cite[6.4]{KM98}.  Since this can be checked locally, this
gives (1).

If $Z=\sp A$ is affine it suffices to check that $R(X,k(K_X+\Delta))$ is a finitely
generated $A$-algebra.  Since the relative Picard number is one, there are real numbers
$a$ and $b$ such that $a(K_X+\Delta)$ and $bS$ are numerically equivalent over $Z$.  As
both $-(K_X+\Delta)$ and $-S$ are ample $\mathbb{Q}$-divisors we may assume that $a$ and
$b$ are both positive integers.  Moreover, as $a(K_X+\Delta)-bS$ is numerically trivial
over $Z$, it is semiample over $Z$ by the base point free theorem.  In particular, we may
replace numerical equivalence by linear equivalence,
$$
a(K_X+\Delta) \sim_Z bS.
$$
But then there is a rational function $g$ and a divisor $D$ on $Z$ such that
$$
a(K_X+\Delta)=bS+f^*D+(g).
$$
As any line bundle on a quasi-projective variety is locally trivial, possibly passing to
an open subset of $Z$, and using (1), we may assume that $D \sim 0$, so that
$$
a(K_X+\Delta) \sim bS.
$$
By \eqref{l_truncation} it follows that $R(X,k(K_X+\Delta))$ is finitely generated if and
only if $R(X,S)$ is finitely generated.  Since $Z$ is affine and $f$ is small, $S$ is
mobile so that $S\sim S'$ where $S'\geq 0$ is a divisor whose support does not contain
$S$.  By \eqref{l_restricted}, $R(X,S)$ is finitely generated if and only if $R_S(X,S')$
is finitely generated.  Since $a(K_X+\Delta)|_S\sim bS'|_S$ the result follows by
\eqref{l_truncation}.  \end{proof}

\section{Multiplier ideal sheaves}
\label{s_squeeze}

The main result of this section is:
\begin{theorem}\label{t_squeeze} Let $\pi\colon\map X.Z.$ be a projective morphism 
to a normal affine variety $Z$, where $(X,\Delta=S+A+B)$ is a log pair, $S=\rdown\Delta.$
is irreducible, $(X,S)$ is log smooth, and both $A\geq 0$ and $B\geq 0$ are
$\mathbb{Q}$-divisors.  Let $k$ be any positive integer and $0\leq \Phi\leq
\Omega=(\Delta-S)|_S$ be any divisor such that both $k(K_S+\Phi)$ and $k(K_X+\Delta)$ are
Cartier.  Let $C=A/k$.

If there is an integer $l>1$ and an integral divisor $P\geq 0$ such that $lA$ is Cartier, 
$C-\frac{(k-1)}mP$ is ample, $(X,\Delta+\frac{k-1}mP)$ is purely log terminal and
$$
l|k(K_S+\Phi)|+m(\Omega-\Phi)+(mC+P)|_S\subset |m(K_X+\Delta+C)+P|_S,
$$ 
where $m=kl$, then 
$$
|k(K_S+\Phi)|+k(\Omega-\Phi)\subset |k(K_X+\Delta)|_S.
$$ 
\end{theorem}

To prove \eqref{t_squeeze}, we need a variant of multiplier ideal sheaves:
\begin{definition-lemma}\label{d_variant} Let $(X,\Delta)$ be a log smooth pair where
$\Delta$ is a reduced divisor and let $V$ be a linear system whose base locus contains no
log canonical centres of $(X,\Delta )$.  Let $\mu\colon\map Y.X.$ be a log resolution of
$V$ and $(X,\Delta )$
and let $F$ be the fixed divisor of the linear system $\mu^*V$.  Let $K_Y+\Gamma =\mu
^*(K_X+\Delta)+ E$ where $\Gamma=\sum P_i$ is the sum of the divisors on $Y$ of log
discrepancy zero.

Then for any real number $c\geq 0$, define the \textbf{multiplier ideal sheaf}
$$
\mathcal{J}_{\Delta,c\cdot V}:=\mu _*\ring Y.( E-\rdown cF.).
$$
If $\Delta=0$ we will write $\mathcal{J}_{c\cdot V}$ and if $D=cG$, where $G>0$ is a
Cartier divisor, we define
$$
\mathcal{J}_{\Delta,D}:=\mathcal{J}_{\Delta,c\cdot V},
$$
where $V=\{G\}$.   
\end{definition-lemma}
\begin{proof} We have to show that the definition of the multiplier ideal sheaf is
independent of the choice of log resolution.  Let $\mu\colon\map Y.X.$ and $\mu'\colon\map
Y'.X.$ be two log resolutions of $(X,\Delta)$ and $V$.  We may assume that $\mu'$
factors through $\mu$ via a morphism $\nu\colon\map Y'.Y.$.  Then $F'=\nu ^* F$ as $\mu ^*V-F$ is
free, and
\begin{align*} 
E'-cF' &= K_{Y'}+\Gamma'-\mu'^*(K_X+\Delta)-cF' \\
       &= K_{Y'}+\Gamma'-\nu^*(K_Y+\Gamma-E+cF) \\ 
       &= \nu^*( E-\rdown cF.)+K_{Y'}+\Gamma'-\nu^*(K_Y+\Gamma+\{cF\}) \\
       &= \nu^*(E -\rdown cF.)+G.
\end{align*} 
Since $(Y,\Gamma+E+F)$ is log smooth, it follows that $(Y,\Gamma+\{cF\})$ is log canonical
and has the same log canonical places as $(Y,\Gamma)$ and hence as $(X,\Delta )$.  Thus
$\rup G.\geq 0$ and since $\nu _*(K_{Y'}+\Gamma ')=K_Y+\Gamma$, $\rup G .$ is
$\nu$-exceptional. Then
\begin{align*}
\mu'_*\ring Y'.(E'-\rdown cF'.)&=\mu_*(\nu_*\ring Y'.(E'-\rdown cF'.)) \\
                               &=\mu_*(\nu_*\ring Y'.(\nu^*(E-\rdown cF.)+\rup G.)) \\
                               &=\mu_*\ring Y.(E-\rdown cF.). \qedhere
\end{align*}
\end{proof}

We need to develop a little of the theory of multiplier ideal sheaves.
\begin{lemma}\label{l_theory} Let $(X,\Delta)$ be a log smooth pair where $\Delta $ is
reduced, let $V$ be a linear system whose base locus contains no log canonical centres
of $(X,\Delta )$ and let $G\geq 0$ and $D\geq 0$ be $\mathbb{Q}$-Cartier divisors whose
supports contain no log canonical centres of $(X,\Delta )$.

 Then
\begin{enumerate} 
\item $\mathcal{J}_{\Delta,D}=\ring X.$ if and only if $(X,\Delta+D)$ is divisorially log terminal 
and $\rdown D.=0$.  
\item If $0\leq \Delta'\leq \Delta$ then $\mathcal{J}_{\Delta,c\cdot V}\subset
\mathcal{J}_{\Delta',c\cdot V}$.  In particular, $\mathcal{J}_{\Delta,c\cdot V}\subset
\mathcal{J}_{c\cdot V} \subset \ring X.$.
\item If $\Sigma\geq 0$ is a Cartier divisor, $D-\Sigma\leq G$ and
$\mathcal{J}_{\Delta,G}=\ring X.$ then $\mathcal{I}_{\Sigma}\subset
\mathcal{J}_{\Delta,D}$.
\end{enumerate} 
\end{lemma}
\begin{proof} (1) follows easily from the definitions.  

(2) follows from the fact that $a(P,X,\Delta')\geq a(P,X,\Delta)$ for all divisors $P$ on
$Y$.

To see (3), notice that as $\Sigma$ is Cartier and $\mathcal{J}_{\Delta,G}=\ring X.$, we have
$$\mathcal{J}_{\Delta,G}(-\Sigma)=\ring X.(-\Sigma)=\mathcal I _{\Sigma}.$$
But since $D \leq G+\Sigma$, we also have 
\[
\mathcal{J}_{\Delta,G}(-\Sigma)=\mathcal{J}_{\Delta,G+\Sigma} \subset \mathcal{J}_{\Delta,D}.\qedhere
\]
\end{proof}

We have the following extension of (9.5.1) of \cite{Lazarsfeld04a} or (2.4.2) of
\cite{Takayama06}:
\begin{lemma}\label{l_ses} Let $\pi\colon\map X.Z.$ be a projective morphism to a 
normal affine variety $Z$.  Let $(X,\Delta)$ be a log smooth pair where $\Delta $ is
reduced, let $S$ be a component of $\Delta$, let $D\geq 0$ be a $\mathbb{Q}$-Cartier
divisor whose support does not contain any log canonical centres of $(X,\Delta)$ and let
$\Theta=(\Delta-S)|_S$.  Let $N$ be a Cartier divisor.  
\begin{enumerate} 
\item There is a short exact sequence
$$
\ses \mathcal{J}_{\Delta-S,D+S}.\mathcal{J}_{\Delta,D}.\mathcal{J}_{\Theta,D|_S}..
$$
\item(Nadel Vanishing) If $N-D$ is ample then 
$$
H^i(X,\mathcal{J}_{\Delta,D}(K_X+\Delta+N))=0,
$$
for $i>0$.
\item If $N-D$ is ample then 
$$
\map
{H^0(X,\mathcal{J}_{\Delta,D}(K_X+\Delta+N))}.{H^0(S,\mathcal{J}_{\Theta,D|_S}(K_X+\Delta+N))}.,
$$ 
is surjective.
\end{enumerate} 
\end{lemma}
\begin{proof} By the resolution lemma of \cite{Szabo94}, we may find a log resolution
$\mu\colon\map Y.X.$ of $(X,\Delta+D)$ which is an isomorphism over the generic point of
each log canonical centre of $(X,\Delta)$.  If $T$ is the strict transform of $S$ then we
have a short exact sequence
$$
\ses {\ring Y.(E-\rdown\mu^*D.-T)}.{\ring Y.(E-\rdown\mu^*D.)}.{\ring T.(E-\rdown\mu^*D.)}..
$$
Now $\mu_*\ring Y.(E-\rdown\mu^*D.)=\mathcal{J}_{\Delta,D}$.  If $\Gamma$ is the sum of
the divisors of log discrepancy zero then
$$
E-\mu^*D=(K_Y+\Gamma)-\mu^*(K_X+\Delta+D).
$$
But then 
$$
E-\mu^*D-T=(K_Y+\Gamma-T)-\mu^*(K_X+\Delta-S+(D+S)), 
$$
so that
$$
\mu_*\ring Y.(E-\rdown\mu^*D.-T)=\mathcal{J}_{\Delta-S,D+S},
$$
and
$$
(E-\mu^*D)|_T=K_T+(\Gamma-T)|_T-\mu^*(K_S+\Theta+D|_S),
$$
so that
$$
\mu_*\ring T.(E-\rdown\mu^*D.)=\mathcal{J}_{\Theta,D|_S}.
$$
Since $(Y,\Gamma+\mu^*D)$ is log smooth and $\Gamma$ and $\mu^*D$ have no common
components, $(Y,\Gamma+\{\mu^*D\})$ is divisorially log terminal.  Therefore we may pick
an exceptional divisor $F\geq 0$ such that $K_Y+\Gamma+\{\mu^*D\}+F$ is divisorially log
terminal and $-F$ is $\mu$-ample.  As
$$
E-\rdown\mu^*D.-T-(K_Y+\Gamma-T+\{\mu ^*D\}+F)=-\mu^*(K_X+\Delta+D)-F,
$$
is $\mu$-ample, Kawamata-Viehweg vanishing implies that
$$
R^1\mu_*\ring Y.(E-\rdown\mu^*D.-T)=0,
$$
and this gives (1).

 Similarly, Kawamata-Viehweg vanishing implies that 
$$
R^i\mu_*\ring Y.(\mu^*(K_X+\Delta+N)+E-\rdown\mu^*D.)=0,
$$
for $i>0$.  As $N-D$ is ample then, possibly replacing $F$ by a small multiple, we may
assume that $\mu^*(N-D)-F$ is ample.  As
\[
\mu^*(K_X+\Delta+N)+E-\rdown\mu^*D.-(K_Y+\Gamma+\{\mu^*D\}+F)=\mu^*(N-D)-F,
\]
is ample, Kawamata-Viehweg vanishing implies that
$$
H^i(Y,\ring Y.(\mu^*(K_X+\Delta+N)+E-\rdown\mu^*D.))=0,
$$
for $i>0$.  Since the Leray-Serre spectral sequence degenerates, this gives (2), and (3)
follows from (2).  \end{proof}

\begin{proof}[Proof of \eqref{t_squeeze}] Since $(X,\Delta +\frac {k-1}mP)$ is purely log
terminal, $(S, \Omega +\frac{k-1}mP|_S)$ is kawamata log terminal and $S$ is contained in
the support neither of $A$ nor of $P$.  If $\Sigma \in |k(K_S+\Phi)|$ then we may pick a
divisor
$$
G\in |m(K_X+\Delta+C)+P|\quad \text{such that}\quad G|_S=l\Sigma+m(\Omega-\Phi+C|_S)+P|_S.
$$ 
Let 
$$
\Lambda =\frac {k-1}mG+B\qquad \text{and}\qquad N=k(K_X+\Delta)-K_X-S.
$$ 
As the support of the $\mathbb{Q}$-divisor $\Lambda\geq 0$ does not contain $S$ and by
assumption
$$
N-\Lambda \sim _{\mathbb Q} C-\frac{k-1}mP,
$$
is ample, \eqref{l_ses} implies that sections of
$H^0(S,\mathcal{J}_{\Lambda|_S}(k(K_S+\Omega)))$ extend to sections of $H^0(X,\ring
X.(k(K_X+\Delta)))$.  Now
\begin{align*} 
&\phantom{\leq } \Lambda |_S -(\Sigma +k(\Omega-\Phi))\\
&=\frac {k-1}m(l\Sigma+m(\Omega-\Phi+C|_S)+P|_S)+B|_S-(\Sigma+k(\Omega-\Phi))  \\ 
&\leq \Omega+\frac {k-1}m P|_S.
\end{align*}
As $(S, \Omega +\frac{k-1}mP|_S)$ is kawamata log terminal, $\mathcal J_{\Omega
  +\frac{k-1}mP|_S}=\ring S.$ and we are done by (3) of \eqref{l_theory}. \end{proof}

\section{Asymptotic multiplier ideal sheaves}
\label{s_variations}

\begin{definition}\label{d_additive} Let $X$ be a normal variety and let $D$ be a divisor.  
An \textbf{additive sequence of linear systems associated to $D$} is a sequence $V\bd$, such that
$V_m\subset \proj {H^0(X,\ring X.(mD))}.$ and 
$$
V_i+V_j\subset V_{i+j}.
$$
\end{definition}

\begin{definition-lemma}\label{d_asymptotic} Suppose that $(X,\Delta)$ is 
log smooth, where $\Delta$ is reduced and let $V\bd$ be an additive sequence of linear
systems associated to a divisor $D$.  Assume that there is an integer $k>0$ such that no
log canonical centre of $(X,\Delta )$ is contained in the base locus of $V_k$.

If $c$ is a positive real number and $p$ and $q$ are positive integers divisible by $k$
then
$$
\mathcal{J}_{\Delta,\frac cp\cdot V_p}\subset \mathcal{J}_{\Delta,\frac cq\cdot V_q} \qquad \text{$\forall q$ divisible by $p$.}
$$
In particular the \textbf{asymptotic multiplier ideal sheaf} of $V\bd$ 
$$
\mathcal{J}_{\Delta,c\cdot V\bd}=\bigcup _{p>0}\mathcal{J}_{\Delta,\frac cp\cdot V_p},
$$
is given by $\mathcal{J}_{\Delta,c\cdot V\bd}=\mathcal{J}_{\Delta,\frac cp\cdot V_p} $,
for $p$ sufficiently large and divisible.  If we take $V_m=|mD|$ the complete linear
system, then define
$$
\mathcal{J}_{\Delta,c\cdot \|D\|}=\mathcal{J}_{\Delta,c\cdot V\bd},
$$
and if $S$ is a component of $\Delta$ and we take $W_m=|mD|_S$, then define
$$ 
\mathcal{J}_{\Theta,c\cdot \|D\|_S}=\mathcal{J}_{\Theta,c\cdot W\bd},
$$
where $\Theta=(\Delta -S)|_S$. 
\end{definition-lemma}
\begin{proof} If $p$ divides $q$ then pick a common log resolution $\mu\colon\map Y.X.$ of
$V_p$, $V_q$ and $(X,\Delta )$ and note that
$$
\frac 1qF_q\leq \frac 1p F_p,
$$
where $F_p$ is the fixed locus of $\mu^*V_p$ and $F_q$ is the fixed locus of
$\mu^*V_q$. Therefore $\mathcal{J}_{\Delta,\frac cp\cdot V_p}\subset
\mathcal{J}_{\Delta,\frac cq\cdot V_q}$.  The equality $\mathcal{J}_{\Delta,c\cdot
  V\bd}=\mathcal{J}_{\Delta,\frac cp\cdot V_p}, $ now follows as $X$ is Noetherian.
\end{proof}

We are now ready to state the main result of this section:
\begin{theorem}\label{t_multiplier} Let $\pi\colon\map X.Z.$ be a projective morphism 
to a normal affine variety $Z$.  Suppose that $(X,\Delta=S+B)$ is log smooth and purely
log terminal of dimension $n$, where $S=\rdown \Delta.$ is irreducible and let $k$ be a
positive integer such that $D=k(K_X+\Delta)$ is integral.  Let $A$ be any ample $\mathbb
Q$-divisor on $X$.  Let $q$ and $r$ be any positive integers such that $Q=qA$ is very
ample, $rA$ is Cartier and $(j-1)K_X+\Xi+rA$ is ample for every Cartier divisor $0\leq
\Xi\leq j\rup\Delta.$ and every integer $1\leq j \leq k+1$.

If the stable base locus of $D$ does not contain any log canonical centre of $(X,\rup
\Delta.)$, then
$$
\mathcal{J}_{\|mD|_S\|}\subset \mathcal{J}_{\Theta,\|mD+P\|_S} \qquad \text{for all} \qquad m\in\mathbb{N},
$$
where $\Theta=\rup B.|_S$, $p=qn+r$ and $P=pA$.  
Moreover, we have $$\pi _* \mathcal J_{\|mD|_S\|}(mD+P)\subset \im\left(
\pi _* \ring X.(mD+P)\to \pi _* \ring S.(mD+P)\right)$$
for all $m\in \mathbb N$.
\end{theorem}

We will need some results about the sheaves $\mathcal{J}_{\Delta,c\cdot V\bd}$, most of
which are easy generalisations of the corresponding facts for the usual asymptotic
multiplier ideal sheaves.
\begin{lemma}\label{l_id} Let $\pi\colon\map X.Z.$ be a projective morphism 
to a normal affine variety $Z$ and let $D$ be a $\mathbb{Q}$-Cartier divisor.  Suppose
that $(X,\Delta)$ is log smooth, $\Delta $ is reduced and the stable base locus of $D$
contains no log canonical centre of $(X,\Delta)$.  Then
\begin{enumerate}
\item for any real numbers $0<c_1\leq c_2$ there is a natural inclusion
$$
\mathcal{J}_{\Delta, c_2 \cdot \|D\|}\subset \mathcal{J}_{\Delta, c_1 \cdot \|D\|}, 
$$
and 
\item if $D$ is Cartier and $S$ is a component of $\Delta$,
then the image of the map
$$
\map {\pi_*\ring X.(D)}.{\pi _*\ring S.(D)}.,
$$
is contained in $\pi_*\mathcal{J}_{\Theta,\|D\|_S}(D)$ where $\Theta =(\Delta -S)|_S$.
\end{enumerate}
\end{lemma}
\begin{proof} (1) is immediate from the definitions.  

Suppose that $D$ is Cartier.  Pick an integer $p$ such that
$$
\mathcal{J}_{\Theta, \|D\|_S}=\mathcal{J}_{\Theta, \frac 1p \cdot |pD|_S},
$$ 
and a log resolution $\mu\colon\map Y.X.$ of $|D|$, $|pD|$ and $(X,\Delta )$.  Let $T$ be
the strict transform of $S$, let $F_1$ be the fixed locus of $\mu^*|D|$ and let $F_p$ be
the fixed locus of $\mu^*|pD|$.  We have
$$
(\pi\circ\mu)_*\ring Y.(\mu^*D-F_1)=\pi_*\ring X.(D)=(\pi\circ\mu)_*\ring Y.(E+\mu^*D).
$$
The first equality follows by definition of $F_1$ and the second follows as $E\geq 0$ is
exceptional.  As there are inequalities
$$
\mu ^*D-F_1\leq \mu ^* D-\rdown F_p/p.\leq E+\mu ^* D-\rdown F_p/p. \leq E+\mu ^* D,
$$
the image of $\pi_*\ring X.(D)$ is equal to the image of 
$$
(\pi\circ\mu)_*\ring Y.(E+\mu ^*D-\rdown F_p/p.).
$$  
Thus the image of $\pi_*\ring X.(D)$ is contained in
\[
(\pi\circ\mu)_* \ring T.(E+\mu^*D-\rdown F_p/p.)=\pi _*\mathcal{J}_{\Theta, \|D\|_S}(D).\qedhere
\]
\end{proof}

\begin{lemma}\label{l_image} Let $\pi\colon\map X.Z.$ be a projective morphism
to a normal affine variety $Z$ and let $D$ be a Cartier divisor.  Suppose that
$(X,\Delta)$ is log smooth and $\Delta$ is reduced. Let $S$ be a component of $\Delta$ and
$\Theta=(\Delta -S)|_S$.  

If $\mathbf{B}_+(D)$ contains no log canonical centres of $(X,\Delta)$ then the image of
the map
$$
\map {\pi_*\ring X.(K_X+\Delta+D)}.{\pi_*\ring S.(K_S+\Theta+D)}.,
$$
contains 
$$
\pi _* \mathcal{J}_{\Theta,\|D\|_S}(K_S+\Theta+D).
$$
\end{lemma}
\begin{proof} Pick an integer $p>1$ such that
$$
\mathcal{J}_{\Theta,\|D\|_S}=\mathcal{J}_{\Theta, \frac 1 p \cdot |pD|_S},
$$
and there is a divisor $A+B\in |pD|$ where $A\geq 0$ is a general very ample divisor and
$B\geq 0$ contains no log canonical centres of $(X,\Delta)$.  By the resolution lemma of
\cite{Szabo94}, we may find a log resolution $\mu\colon\map Y.X.$ of $|pD|$ and of
$(X,\Delta)$ which is an isomorphism over every log canonical centre of $(X,\Delta)$.  Let
$F_p$ be the fixed divisor of $\mu^*|pD|$, $M_p=p\mu ^*D-F_p$ and let $\Gamma$ and $T$ be
the strict transforms of $\Delta$ and $S$.  We have a short exact sequence
$$
\ses {\ring Y.(G-T)}.{\ring Y.(G)}.{\ring T.(G)}.,
$$
where $G=K_Y+\Gamma+\mu ^*D-\rdown F_p/p.$.  As $\mu^*A$ is base point free and
$\mu^*(A+B)\in \mu^*|pD|$, the divisor $C:=\mu^*B-F_p$ is effective.  Note that $M_p-C
\sim \mu^*A$.  As no component of $C$ is a component of $\Gamma$, we may pick
$0<\delta\leq 1/p$ and an exceptional $\mathbb{Q}$-divisor $F\geq 0$ such that
$(Y,\Gamma-T+\{F_p/p\}+\delta (C+F))$ is divisorially log terminal and $\mu^*A-F$ is
ample.  As $|M_p|$ is free, $M_p/p$ is nef and so
\begin{align*} 
G-T-(K_Y+\Gamma-T+\{\frac 1pF_p\}+\delta(C+F))&=\frac 1pM_p-\delta(C+F) \\
                                              &\sim_{\mathbb{Q}}(\frac 1p-\delta)M_p+\delta(\mu^*A-F),
\end{align*} 
is ample.  In particular Kawamata-Viehweg vanishing implies that $R^1\phi_*\ring
Y.(G-T)=0$ where $\phi=\pi\circ\mu$. Therefore the homomorphism
\[
\pi_*\ring X.(K_X+\Delta+D)\supset\phi_*\ring Y.(G)\longrightarrow\phi_*\ring T.(G)=\pi _*\mathcal{J}_{\Theta,\|D\|_S}(K_S+\Theta+D),
\] is surjective.  \end{proof}

\begin{theorem}\label{t_generation} Let $\pi\colon\map X.Z.$ be a projective morphism,
where $Z$ is affine and $X$ is a smooth variety of dimension $n$.  

If $D$ is a Cartier divisor whose stable base locus is a proper subset of $X$, $A$ is an
ample Cartier divisor and $H$ is a very ample divisor then
$\mathcal{J}_{\|D\|}(D+K_X+A+nH)$ is globally generated.
\end{theorem}
\begin{proof} Pick an integer $p>0$ such that if $pB\in |pD|$ is a general element,
then 
$$
\mathcal{J}_{\|D\|}=\mathcal{J}_{\frac 1p \cdot |pD|}=\mathcal{J}_{B}.
$$
Then by (2) of \eqref{l_ses}, $H^i(X,\mathcal{J}_{\|D\|}(D+K_X+A+mH))=0$ for all $i>0$ and
$m\geq 0$ and we may apply \eqref{l_generation}.  \end{proof}

\begin{lemma}\label{l_generation} Let $\pi\colon\map X.Z.$ be a projective morphism
where $X$ is smooth of dimension $n$, $Z$ is affine and let $H$ be a very ample divisor.

If $\mathcal{F}$ is any coherent sheaf such that $H^i(X,\mathcal{F}(mH))=0$, for $i>0$ and
for all $m\geq -n$ then $\mathcal{F}$ is globally generated.
\end{lemma}
\begin{proof} Pick $x\in X$.  Let $\mathcal{T}\subset\mathcal{F}$ be the torsion subsheaf
supported at $x$, and let $\mathcal{G}=\mathcal{F}/\mathcal{T}$.  Then
$H^i(X,\mathcal{G}(mH))=0$ for $i>0$ and for all $m\geq -n$ and $\mathcal{F}$ is globally
generated if and only if $\mathcal{G}$ is globally generated.  Replacing $\mathcal{F}$ by
$\mathcal{G}$ we may therefore assume that $\mathcal{T}=0$.

Pick a general element $Y\in |H|$ containing $x$.  As $\mathcal{T}=0$ there is an exact
sequence
$$
\ses \mathcal{F}(-Y).\mathcal{F}.\mathcal{G}.,
$$
where $\mathcal{G}=\mathcal{F}\otimes \ring Y.$.  As $H^i(Y,\mathcal{G}(mH))=0$, for $i>0$
and for all $m\geq -(n-1)$, $\mathcal{G}$ is globally generated by induction on the
dimension.  As $H^1(X,\mathcal{F}(-Y))=0$ it follows that $\mathcal{F}$ is globally
generated.  \end{proof}

\begin{proof}[Proof of \eqref{t_multiplier}] We follow the argument of \cite{HM05b} which
in turn is based on the ideas of \cite{Kawamata99}, \cite{Siu98} and \cite{Tsuji99}.

We proceed by induction on $m$.  The statement is clear for $m=0$, and so it suffices to
show that
$$
\mathcal{J}_{\|(m+1)D|_S\|}\subset \mathcal{J}_{\Theta,\|(m+1)D+P\|_S},
$$
assuming that 
$$
\mathcal{J}_{\|tD|_S\|}\subset \mathcal{J}_{\Theta,\|tD+P\|_S} \qquad \text{for all} \qquad t\leq m.  
$$
If $\Delta=\sum \delta _i\Delta _i$, where each $\Delta_i$ is a prime divisor, then for
any $1\leq s\leq k$, put
$$
\Delta^s=\sum _{i\,:\,\delta _i>(k-s)/k}\Delta _i  .
$$
We have 
\begin{itemize} 
\item each $\Delta^s$ is integral, 
\item $\ds{S=\Delta^1\leq \Delta^2 \leq \cdots \leq \Delta^k=\rup \Delta.}$, and
\item $\ds{\Delta=\frac 1k\sum_{s=1}^k \Delta^s}$,
\end{itemize} 
and these properties uniquely determine the divisors $\Delta^s$.  We let
$\Delta^{k+1}=\rup \Delta.$.  We recursively define integral divisors $D_{\leq s}$ by the
rule
$$
D_{\leq s} = \begin{cases} 
          0               & \text{if $s=0$} \\
 K_X+\Delta^s+D_{\leq s-1} &  1\leq s\leq k.
\end{cases}
$$
Note that $D_{\leq k}=D$.  By (1) of \eqref{l_id} there is an inclusion
$$
\mathcal{J}_{\|(m+1)D|_S\|}\subset \mathcal{J}_{\|mD|_S\|},
$$
and so it suffices to prove that there are inclusions
\[
\label{e_star}\mathcal{J}_{\|mD|_S\|}\subset \mathcal{J}_{\Theta^{s+1},\|mD+D_{\leq s}+P\|_S}, \tag{$\star$} 
\]
for $0\leq s\leq k$, where $\Theta^i=(\Delta^i-S)|_S$ for $1\leq i\leq k+1$.  
Thus $\Theta^{k}=\Theta^{k+1}=\Theta$ and $\Theta^1=0$.  

We proceed by induction on $s$. Now
$$
\mathcal{J}_{\|mD|_S\|}\subset \mathcal{J}_{\Theta,\|mD+P\|_S} \subset \mathcal{J}_{\Theta^1,\|mD+P\|_S}.
$$
The first inclusion holds by assumption and since $\Theta^1\leq \Theta$, (2) of
\eqref{l_theory} implies the second inclusion.  Thus \eqref{e_star} holds when $s=0$.

Now suppose that \eqref{e_star} holds for $s\leq t-1$.  Note that 
\begin{align*}
\label{e_dagger} mD+D_{\leq t}+P&=K_X+\Delta^t+(D_{\leq t-1}+P)+mD\\
                               &=mD+K_X+(\Delta^t+D_{\leq t-1}+rA)+nQ,\tag{$\dagger$}
\end{align*}
where, by assumption, both $D_{\leq t-1}+P$ and $\Delta^t+D_{\leq t-1}+rA$ are ample for
any $1\leq t\leq k+1 $.  In particular $\mathbf{B}_{+}(mD+D_{\leq t-1}+P)$ contains no log
canonical centres of $(X,\rup \Delta.)$.  Then
\begin{align*}
\pi_*\mathcal{J}_{\|mD|_S\|}(mD+D_{\leq t}+P)
&\subset \pi _*\mathcal{J}_{\Theta^t,\|mD+D_{\leq t-1}+P\|_S}(mD+D_{\leq t}+P)\\
&\subset \im \left( \pi _* \ring X.(mD+D_{\leq t}+P)\longrightarrow \pi _*\ring S.(mD+D_{\leq t}+P)\right)\\
&\subset \pi _* \mathcal{J}_{\Theta^{t+1},\|mD+D_{\leq t}+P\|_S}(mD+D_{\leq t}+P).
\end{align*} 
The first inclusion holds as we are assuming \eqref{e_star} for $s=t-1$, the second
inclusion holds by \eqref{e_dagger} and \eqref{l_image} and the last inclusion follows
from (2) of \eqref{l_id}.  But \eqref{e_dagger} and \eqref{t_generation} imply that
$$
\mathcal{J}_{\|mD|_S\|}(mD+D_{\leq t}+P),
$$
is generated by global sections and so
\[
\mathcal{J}_{\|mD|_S\|}\subset \mathcal{J}_{\Theta^{t+1}, \|mD+D_{\leq t}+P\|_S}. 
\] 
The inclusion 
$$
\pi _* \mathcal J_{\|mD|_S\|}(mD+P)\subset \im\left(\pi _* \ring X.(mD+P)\to \pi _*
\ring S.(mD+P)\right),
$$ 
is part of the inclusions proved above when $s=k$.  
\end{proof}

\section{Lifting sections}
\label{s_lift}

\begin{lemma}\label{l_lim} Let $D\geq 0$ be a Cartier divisor on a normal variety
$X$, and let $Z\subset X$ be an irreducible subvariety. 

Then
$$
\liminf \frac{\mult_Z(|mD|)}m=\lim \frac{\mult_Z(|m!D|)}{m!}.
$$
\end{lemma}
\begin{proof} Note that if $a$ divides $b$ then 
$$
\frac{\mult_Z(|aD|)}a\geq \frac{\mult_Z(|bD|)}b,
$$
whence the result.  \end{proof}
\begin{lemma}\label{l_achieve} Let $D\subset X$ be a divisor on a smooth variety and $Z$ 
a closed subvariety.  

If $\lim\mult_Z(|m!D|)/m!=0$ then $Z$ is not contained in $\mathbf{B}_-(D)$.
\end{lemma}
\begin{proof} Let $A$ be any ample divisor.  Pick $l>0$ such that $lA-K_X$ is ample.  If
$m>l$ is sufficiently divisible then $\mathcal{J}_{\|mD\|}(m(D+A))$ is globally generated
by \eqref{t_generation}.  But if $p>0$ is sufficiently large and divisible and $D_{mp}\in
|mpD|$ is general, then $\mult_Z D_{mp}=\mult_Z|mpD|<p$ and
$$
\mathcal J_{\|mD\|}=\mathcal J_{(1/p)D_{mp}}. 
$$ 
But since $\mult_Z D_{mp}/p<1$ it follows that $(X,D_{mp}/p)$ is 
kawamata log terminal, in a neighbourhood of the generic point of $Z$.  Thus $Z$ is not
contained in the co-support of $\mathcal J_{\|mD\|}$ and so $Z$ is not contained in the
base locus of $m(D+A)$.  \end{proof}

\begin{theorem}\label{t_lift} Let $\pi\colon\map X.Z.$ be a projective morphism 
to a normal affine variety $Z$, where $(X,\Delta=S+A+B)$ is a purely log terminal pair,
$S=\rdown\Delta.$ is irreducible, $(X,S)$ is log smooth, $A\geq 0$ is a general ample
$\mathbb{Q}$-divisor, $B\geq 0$ is a $\mathbb{Q}$-divisor and $(S,\Omega+A|_S)$ is
canonical, where $\Omega=(\Delta-S)|_S$.  Assume that the stable base locus of
$K_X+\Delta$ does not contain $S$.  Let $F=\lim F_{l!}$, where, for any positive and
sufficiently divisible integer $m$, we let
\[
F_m=\fix (|m(K_X+\Delta)|_S)/m . 
\]
If $\epsilon>0$ is any rational number such that $\epsilon(K_X+\Delta)+A$ is ample and if
$\Phi$ is any $\mathbb{Q}$-divisor on $S$ and $k>0$ is any integer such that
\begin{enumerate}
\item both $k\Delta$ and $k\Phi$ are Cartier, and
\item $\Omega\wedge \lambda F \leq \Phi\leq \Omega$, where $\lambda=1-\epsilon/k$, 
\end{enumerate}
then
$$
|k(K_S+\Omega-\Phi)|+k\Phi\subset |k(K_X+\Delta)|_S.
$$
\end{theorem}
\begin{proof} By assumption $A=H/m$, where $H$ is very ample and a very 
general element of $|H|$ and $m\geq 2$ is an integer.  If $C=A/k$, then 
$$
A+(k-1)C=\frac{2k-1}{km}H,
$$
and so 
$$
(X,\Delta+(k-1)C=S+\frac{2k-1}{km}H+B) 
$$
is purely log terminal, as 
$$
\frac{2k-1}{km}<1.
$$
On the other hand, 
$$
(S,\Omega+C|_S),
$$
is canonical as we are even assuming that $(S,\Omega+A|_S)$ is canonical.  Pick
$\eta>\epsilon/k$ rational so that $\eta(K_X+\Delta)+C$ is ample and let
$\mu=1-\eta<\lambda=1-\epsilon /k$.  If $l>0$ is any sufficiently divisible integer so
that $O=l(\eta(K_X+\Delta)+C)$ is very ample, then
\begin{align*} 
G_l &=\fix( |l(K_X+\Delta+C)|_S)/l \\ 
    &=\fix(|l\mu (K_X+\Delta)+O|_S)/l \\ 
    &\leq \fix(|l\mu (K_X+\Delta)|_S)/l \\
    &=\mu F_{\mu l}.
\end{align*} 
Thus
$$
\lim G_{l!}\leq\mu\lim F_{l!}=\mu F. 
$$
On the other hand \eqref{l_achieve} implies that there is a positive integer $l$ such that
every prime divisor on $S$ which does not belong to the support of $F$ does not belong to
the base locus of $|l(K_X+\Delta+C)|$.  Thus we may pick a positive integer $l$ such that
\begin{itemize}  
\item $k$ divides $l$, 
\item $lC$ is Cartier, and 
\item $G_l\leq \lambda F$.  
\end{itemize} 

Let $f\colon\map Y.X.$ be a log resolution of the linear system $|l(K_X+\Delta+C)|$ and of
$(X,\Delta+C)$.  We may write
$$
K_Y+\Gamma=f^*(K_X+\Delta+C)+E,
$$
where $\Gamma\geq 0$ and $E\geq 0$ have no common components, $f_*\Gamma =\Delta+C$ and
$f_*E=0$.  Then 
$$
H_l=\fix(l(K_Y+\Gamma))/l=\fix(lf^*(K_X+\Delta+C))/l+E.
$$

If $\Xi=\Gamma-\Gamma\wedge H_l$ then $l(K_Y+\Xi)$ is Cartier and $\fix(l(K_Y+\Xi))$ and
$\Xi$ share no common components.  Since the mobile part of $|l(K_Y+\Xi)|$ is free and the
support of $\fix(l(K_Y+\Xi))+\Xi$ has normal crossings it follows that the stable base
locus of $K_Y+\Xi$ contains no log canonical centres of $(Y,\rup \Xi.)$ (which are nothing
but the strata of $\rup\Xi.$). 

Let $H\geq 0$ be any ample divisor on $Y$.  Pick positive integers $m$ and $q$ such that
$l$ divides $m$ and $Q=qH$ is very ample.  Let $T$ be the strict transform of $S$, let
$\Gamma _T=(\Gamma-T)|_T$ and let $\Xi_T=(\Xi-T)|_T$.  If
$$
\tau\in H^0(T,\ring T.(m(K_T+\Xi _T)))=H^0(T,\mathcal{J}_{\|m(K_T+\Xi_T)\|}(m(K_T+\Xi _T))),
$$
and $\sigma\in H^0(T,\ring T.(Q))$  then 
$$
\sigma\cdot \tau \in H^0(T,\mathcal{J}_{\|m(K_T+\Xi _T)\|}(m(K_T+\Xi _T)+Q)).
$$
On the other hand, if $q$ is sufficiently large and divisible then by \eqref{t_multiplier}
$H^0(T,\mathcal{J}_{\|m(K_T+\Xi_T)\|}(m(K_T+\Xi _T)+Q))$ is contained in the image of
\[
\map {H^0(Y,\ring Y.(m(K_Y+\Xi)+Q))}.{H^0(T,\ring T.(m(K_T+\Xi _T)+Q))}..
\]
Hence there is a fixed $q$ such that whenever $l$ divides $m$, we have
\[
|m(K_T+\Xi _T)|+m(\Gamma _T-\Xi _T)+|Q|_T|\subset |m(K_Y+\Gamma)+Q|_T. 
\]
If $g=f|_T\colon \map T.S.$ then $g_*\Gamma _T=\Omega+C|_S$ and since $g_*\Xi _T\leq
\Omega+C|_S$ and $(S, \Omega+C|_S)$ is canonical, we have $|m(K_S+g_*\Xi_T)|=g_*|m(K_T+\Xi _T)|$.
Therefore, applying $g_*$, we obtain
\[
|m(K_S+g_*\Xi _T)|+m(\Omega+C|_S-g_*\Xi _T)+P|_S\subset |m(K_X+\Delta +C)+P|_S,
\]
where $P=f_*Q$.  

Since for every prime divisor $L$ on $S$ we have
$$
\mult_L G_l=\mult_{L'}\fix(|l(K_Y+\Gamma)|_T)/l=\mult_{L'} H_l|_T,
$$
where $L'$ is the strict transform of $L$ on $T$, it follows that 
$$
g_*\Xi_T-C|_S=\Omega-\Omega \wedge G_l\geq \Omega-\Omega \wedge \lambda F
\geq\Omega-\Phi\geq 0.
$$
Therefore
\[
|m(K_S+\Omega-\Phi)|+m\Phi+(mC+P)|_S \subset |m(K_X+\Delta+C)+P|_S,
\]
for any $m$ divisible by $l$.  In particular if we pick $m$ so that 
$C-\frac{k-1}mP$ is ample and $(X,\Delta+\frac{k-1}mP)$ is purely log terminal 
then the result follows by \eqref{t_squeeze}.  \end{proof}

\section{Rationality of the restricted algebra}
\label{s_rationality}

In this section we will prove:
\begin{theorem}\label{t_rational} Assume Theorem~\ref{t_ezd}$_{n-1}$. 

Let $\pi\colon\map X.Z.$ be a projective morphism to a normal affine variety $Z$, where
$(X,\Delta=S+A+B)$ is a purely log terminal pair of dimension $n$, $S=\rdown\Delta.$ is
irreducible, $(X,S)$ is log smooth, $A\geq 0$ is a general ample $\mathbb{Q}$-divisor,
$B\geq 0$ is a $\mathbb{Q}$-divisor and $(S,\Omega+A|_S)$ is canonical, where
$\Omega=(\Delta-S)|_S$.  Assume that the stable base locus of $K_X+\Delta$ does not
contain $S$.  Let $F=\lim F_{l!}$ where, for any positive and sufficiently divisible
integer $m$, we let
\[
F_m=\fix (|m(K_X+\Delta)|_S)/m.
\]

Then $\Theta=\Omega-\Omega\wedge F$ is rational.  In particular if both $k\Delta$ and
$k\Theta$ are Cartier then
$$
|k(K_S+\Theta)|+k(\Omega-\Theta)=|k(K_X+\Delta)|_S,
$$
and
$$
R_S(X,k(K_X+\Delta))\simeq R(S,k(K_S+\Theta)).
$$
\end{theorem}
\begin{proof} Suppose that $\Theta$ is not rational.  Let $V\subset\WDiv_{\mathbb{R}}(S)$
be the vector space spanned by the components of $\Theta$.  Then there is a constant
$\delta>0$ such that if $\Phi\in V$ and $\|\Phi-\Theta\| <\delta$ then $\Phi\geq 0$ has
the same support as $\Theta$ and moreover, by (2) of Theorem~\ref{t_ezd}$_{n-1}$, if $G$
is a prime divisor contained in the stable base locus of $K_S+\Theta$ then it is also
contained in the stable base locus of $K_S+\Phi$.

If $l(K_X+\Delta)$ is Cartier and $\Theta_l=\Omega-\Omega\wedge F_l$ then 
$$
|l(K_X+\Delta)|_S\subset |l(K_S+\Theta_l)|+l(\Omega\wedge F_l).
$$
Hence $\fix (l(K_S+\Theta_l))$ does not contain any components of $\Theta_l$.  In
particular the stable base locus of $K_S+\Theta_l$ does not contain any components of
$\Theta_l$.  But we may pick $l>0$ so that $\Theta_l\in V$ and
$\|\Theta_l-\Theta||<\delta$.  It follows that no component of $\Theta$ is in the stable
base locus of $K_S+\Theta$.

Let $W\subset V$ be the smallest rational affine space which contains $\Theta$.
(3) of Theorem~\ref{t_ezd}$_{n-1}$ implies that there is a positive integer $r>0$ and a positive
constant $\eta>0$ such that if $\Phi\in W$, $k\Phi/r$ is Cartier and
$\|\Phi-\Theta\|<\eta$ then every component of $\fix (k(K_S+\Phi))$ is in fact a component
of the stable base locus of $K_S+\Theta$.

Pick a rational number $\epsilon>0$ such that $\epsilon(K_X+\Delta)+A$ is ample.  By
Diophantine approximation, we may find a positive integer $k$, a divisor $\Phi$ on $S$ and
a prime divisor $G$ (necessarily a component of $\Theta$ whose coefficient is irrational)
such that
\begin{enumerate}
\item $0\leq \Phi\in W$, 
\item both $k\Phi/r$ and $k\Delta/r$ are Cartier, 
\item $\|\Phi-\Theta\|<\min(\delta,\eta,f\epsilon/k)$ where $f$ is the smallest non-zero
coefficient of $F\neq 0$, and
\item $\mult_G\Phi>\mult_G \Theta$. 
\end{enumerate} 

\begin{claim}\label{c_component} $\Omega\wedge \lambda F\leq \Omega-\Phi \leq \Omega$,
where $\lambda=1-\epsilon/k$.   
\end{claim}
\begin{proof}[Proof of \eqref{c_component}] Let $P$ be a prime divisor on $S$ and let
$\omega$, $f$, $\phi$ and $\theta$ be the multiplicities of $\Omega$, $F$, $\Phi$ and
$\Theta$ along $P$.  We just need to check that 
\[
\label{e_component}\min(\omega,\lambda f)\leq \omega-\phi. \tag{$*$}
\]

There are two cases.  If $\omega \leq f$, then $\theta=0$ so that $\phi =0$ and
\eqref{e_component} holds.  If $\omega \geq f$, then $\theta =\omega -f$ and since
$\|\Phi-\Theta\|<f\epsilon/k$,
\[
\min (\omega ,\lambda f)=\left(1-\frac \epsilon k\right)f \leq f-(\phi -\theta )=\omega -\phi. \qedhere
\]
\end{proof}

\eqref{c_component}, (2) and \eqref{t_lift} imply that 
$$
|k(K_S+\Phi)|+k(\Omega-\Phi)\subset |k(K_X+\Delta)|_S.
$$
(4) implies that $G$ is a component of $\fix(k(K_S+\Phi))$.  (2) and
$\|\Phi-\Theta\|<\eta$ imply that $G$ is a component of the stable base locus of
$K_S+\Theta$, a contradiction.

Thus $\Theta$ is rational.  Hence $\Omega\wedge F$ is rational, and we are done by
\eqref{t_lift}.  \end{proof}

\section{Proof of \eqref{t_m}}
\label{s_pl}

\begin{theorem}\label{t_model} Assume Theorem~\ref{t_ezd}$_{n-1}$.

Let $\pi\colon\map X.Z.$ be a projective morphism to a normal affine variety $Z$.  Suppose
that $(X,\Delta=S+A+B)$ is a purely log terminal pair of dimension $n$, $S=\rdown \Delta.$
is irreducible and not contained in the stable base locus of $K_X+\Delta$, $A\geq 0$ is a
general ample $\mathbb{Q}$-divisor and $B\geq 0$ is a $\mathbb{Q}$-divisor.

Then there is a birational morphism $g\colon\map T.S.$, a positive integer $l$ and a
kawamata log terminal pair $(T,\Theta)$ such that $K_T+\Theta$ is $\mathbb{Q}$-Cartier and
$$
R_S(X,l(K_X+\Delta))\cong R(T,l(K_T+\Theta)).
$$
\end{theorem}
\begin{proof} If $f\colon\map Y.X.$ is a log resolution of $(X,\Delta)$ then we may write
$$
K_Y+\Gamma'=f^*(K_X+\Delta)+E,
$$
where $\Gamma'\geq 0$ and $E\geq 0$ have no common components, $f_*\Gamma'=\Delta$ and
$f_*E=0$.  If $T$ is the strict transform of $S$ then we may choose $f$ so that
$(T,\Psi'=(\Gamma'-T)|_T)$ is terminal.  Note that $T$ is not contained in the stable base
locus of $K_Y+\Gamma'$ as $S$ is not contained in the stable base locus of $K_X+\Delta$.

Pick a $\mathbb{Q}$-divisor $F$ such that $f^*A-F$ is ample and $(Y,\Gamma'+F)$ is purely
log terminal.  Pick $m>1$ so that $m(f^*A-F)$ is very ample and pick $mC\in |m(f^*A-F)|$
very general.  Then 
$$
(Y,\Gamma=\Gamma'-f^*A+F+C \sim_{\mathbb{Q}}\Gamma'),
$$
is purely log terminal and if $m$ is sufficiently large $(T,\Psi+C|_T)$ is terminal, where
$\Psi=(\Gamma-T)|_T$.

On the other hand
\begin{align*} 
R(X,k(K_X+\Delta))&\cong R(Y,k(K_Y+\Gamma))  \qquad \text{and}\\ 
R_S(X,k(K_X+\Delta))&\cong R_T(Y,k(K_Y+\Gamma)),
\end{align*} 
for any $k$ sufficiently divisible.  Now apply \eqref{t_rational} to $(Y,\Gamma)$.  \end{proof}

\begin{proof}[Proof of \eqref{t_m}] By \eqref{t_restricted} we may assume that $Z$ is
affine and by \eqref{l_restricted}, it suffices to prove that the restricted algebra is
finitely generated.  As $Z$ is affine, $S$ is mobile and as $f$ is birational, the divisor
$\Delta-S$ is big.  But then
$$
\Delta-S\sim_{\mathbb{Q}} A+B,
$$ 
where $A$ is a general ample $\mathbb{Q}$-divisor and $B\geq 0$.  As $S$ is mobile, we may
assume that the support of $B$ does not contain $S$.  Now
$$
K_X+\Delta'=K_X+S+(1-\epsilon)(\Delta-S)+ \epsilon A+\epsilon B\qle K_X+\Delta, 
$$
is purely log terminal, where $\epsilon$ is any sufficiently small positive rational
number. By \eqref{l_truncation}, we may replace $\Delta$ by $\Delta'$. We may therefore
assume that $\Delta=S+A+B$, where $A$ is a general ample $\mathbb{Q}$-divisor and $B\geq
0$.  Since we are assuming Theorem~\ref{t_ezd}$_{n-1}$, \eqref{t_model} implies that the
restricted algebra is finitely generated.\end{proof}

\bibliographystyle{hamsplain}
\bibliography{/home/mckernan/Jewel/Tex/math}

\end{document}